\documentclass{amsart}

\usepackage{amssymb,amsthm,bm,mathrsfs,hyperref}

\newcommand{\subconstituent}[1]{\Omega_{#1}}
\newcommand{\harm}[1]{U_{#1}}

\newcommand{\gauss}[2]{\genfrac{[}{]}{0pt}{}{#1}{#2}}
\newcommand{\dgauss}[2]{{ \displaystyle \genfrac{[}{]}{0pt}{}{#1}{#2} }}

\theoremstyle{plain}
\newtheorem{thm}{Theorem}[section]
\newtheorem{lem}[thm]{Lemma}
\newtheorem{prop}[thm]{Proposition}

\theoremstyle{remark}
\newtheorem{rem}[thm]{Remark}

\title[A cross-intersection theorem for vector spaces]{A cross-intersection theorem for vector spaces based on semidefinite programming}
\author{Sho Suda}
\address{Department of Mathematics, Aichi University of Education, Kariya 448-8542, Japan}
\email{suda@auecc.aichi-edu.ac.jp}
\author{Hajime Tanaka}
\address{Research Center for Pure and Applied Mathematics, Graduate School of Information Sciences, Tohoku University, Sendai 980-8579, Japan}
\email{htanaka@m.tohoku.ac.jp}

\begin{document}

\begin{abstract}
Let $\mathscr{F}$ and $\mathscr{G}$ be families of $k$- and $\ell$-dimensional subspaces, respectively, of a given $n$-dimensional vector space over a finite field $\mathbb{F}_q$.
Suppose that $x\cap y \ne 0$ for all $x\in \mathscr{F}$ and $y\in \mathscr{G}$.
By explicitly constructing optimal feasible solutions to a semidefinite programming problem which is akin to Lov\'{a}sz's theta function, we show that $|\mathscr{F}| |\mathscr{G}| \leqslant \gauss{n-1}{k-1} \gauss{n-1}{\ell-1}$, provided that $n\geqslant 2k$ and $n\geqslant 2\ell$.
The characterization of the extremal families is also established.
\end{abstract}

\maketitle

\section{Introduction}

Let $V$ be an $n$-dimensional vector space over a finite field $\mathbb{F}_q$.
Let $\subconstituent{}$ be the set of subspaces of $V$.
For $k=0,\dots,n$, let $\subconstituent{k}$ be the subset of $\subconstituent{}$ consisting of the $k$-dimensional subspaces of $V$.
Two families $\mathscr{F}\subseteq \subconstituent{k}$ and $\mathscr{G} \subseteq \subconstituent{\ell}$ are said to be \emph{cross-intersecting} if $x\cap y \ne 0$ for all $x\in \mathscr{F}$ and $y\in \mathscr{G}$.
In the present paper, we apply semidefinite programming to prove the following theorem:

\begin{thm}\label{cross-intersection theorem}
Let $\mathscr{F}\subseteq \subconstituent{k}$ and $\mathscr{G} \subseteq \subconstituent{\ell}$ be a pair of cross-intersecting families.
Suppose that $n\geqslant 2k$ and $n\geqslant 2\ell$.
Then\footnote{Here $\gauss{a}{b}=\gauss{a}{b}_q$ denotes the Gaussian coefficient: $\gauss{a}{b}=\prod_{i=0}^{b-1}((q^{a-i}-1)/(q^{b-i}-1))$.}
\begin{equation*}
	|\mathscr{F}| |\mathscr{G}| \leqslant \gauss{n-1}{k-1} \gauss{n-1}{\ell-1},
\end{equation*}
and equality holds if and only if either (i) there is $z \in \subconstituent{1}$ such that $\mathscr{F}=\{ x \in \subconstituent{k} : z \subseteq x \}$ and $\mathscr{G}=\{ x \in \subconstituent{\ell} : z \subseteq x \}$, or (ii) $n=2k=2\ell$ and there is $z\in \subconstituent{2k-1}$ such that $\mathscr{F}=\mathscr{G}=\{ x \in \subconstituent{k} : x\subseteq z\}$.
\end{thm}

\noindent
This is a $q$-analogue of a theorem about cross-intersecting families of subsets of a given $n$-set, established by Pyber \cite{Pyber1986JCTA} and Matsumoto and Tokushige \cite{MT1989JCTA}.\footnote{The case $k=\ell$ in Theorem \ref{cross-intersection theorem} has recently been settled by Tokushige \cite{Tokushige2012pre} by using the eigenvalue method.}
Their proofs are based on the Kruskal--Katona theorem (cf.~\cite{FT1994P}).
We may remark that a $q$-analogue of the Kruskal--Katona theorem has been obtained, but only in a weaker form which is not yet sufficient to prove Theorem \ref{cross-intersection theorem} (cf.~\cite{CP2010JCTA}).

Linear programming, in particular its duality theory, has been a powerful tool to study various combinatorial configurations.
One of the pioneers is Delsarte \cite{Delsarte1973PRRS}.
His famous LP bound for cliques in association schemes is closely related to Lov\'{a}sz's theta function bound \cite{Lovasz1979IEEE} on the Shannon capacities of graphs, but the corresponding SDP problems are drastically reduced to LP by making full use of the regularity of association schemes; cf.~\cite{Schrijver1979IEEE}.
Indeed, it is this method that made it possible to prove the Erd\H{o}s--Ko--Rado theorem \cite{EKR1961QJMO} as well as its $q$-analogue in full generality; cf.~\cite{Wilson1984C,Tanaka2006JCTA,Tanaka2012C}.
Theorem \ref{cross-intersection theorem} is an extension of the $q$-Erd\H{o}s--Ko--Rado theorem, and we shall consider a `bipartite variant' of an SDP problem defining the theta function.
It involves semidefinite constraints of $2\times 2$ matrices even after a reduction, but we are still able to construct optimal feasible solutions to obtain the bound, and then the duality of SDP again provides enough information to describe the extremal families.
Recently there have been several attempts to extend Delsarte's theory by shifting from LP to SDP (cf.~\cite{Schrijver2005IEEE}),\footnote{Delsarte's LP bound is defined on a commutative semisimple $\mathbb{C}$-algebra associated with an association scheme, called the \emph{Bose--Mesner algebra}. In the process of shifting from LP to SDP, we are also shifting to noncommutative semisimple $\mathbb{C}$-algebras, such as the \emph{Terwilliger algebra} \cite{Terwilliger1992JAC} or a \emph{coherent algebra} \cite{Higman1987LAA}.} but it seems that Theorem \ref{cross-intersection theorem} is among the most successful applications of the duality of SDP to the analysis of optimal combinatorial configurations.

\section{A semidefinite programming problem}

Let $\mathbb{R}^{\subconstituent{}\times \subconstituent{}}$ be the set of real matrices with rows and columns indexed by $\subconstituent{}$, and let $\mathbb{R}^{\subconstituent{}}$ be the set of real column vectors with coordinates indexed by $\subconstituent{}$.
We shall write $Y \bullet Z = \mathrm{trace} (Y^{\mathsf{T}} Z)$ for $Y,Z\in \mathbb{R}^{\subconstituent{}\times \subconstituent{}}$.
Let $S\mathbb{R}^{\subconstituent{}\times \subconstituent{}}$ be the set of symmetric matrices in $\mathbb{R}^{\subconstituent{}\times \subconstituent{}}$.
For $k,\ell=0,\dots,n$, we define $W_{k,\ell}, \overline{W}_{\! k,\ell}\in\mathbb{R}^{\subconstituent{}\times\subconstituent{}}$ by
\begin{align*}
	(W_{k,\ell})_{x,y} &= \begin{cases} 1 & \text{if} \ x\in\subconstituent{k}, \ y\in\subconstituent{\ell}, \ x\cap y\in\subconstituent{\min\{k,\ell\}}, \\ 0 & \text{otherwise}, \end{cases} \\
	(\overline{W}_{\! k,\ell})_{x,y} &= \begin{cases} 1 & \text{if} \ x\in\subconstituent{k}, \ y\in\subconstituent{\ell}, \ x\cap y\in\subconstituent{0}, \\ 0 & \text{otherwise}, \end{cases}
\end{align*}
for $x,y\in \subconstituent{}$.
Then $W_{k,\ell}=W_{\ell,k}^{\mathsf{T}}$ and $\overline{W}_{\! k,\ell}=\overline{W}_{\! \ell,k}^{\mathsf{T}}$.
Notice also that $I_k := W_{k,k}$ is the orthogonal projection matrix onto $\mathbb{R}^{\subconstituent{k}} \subseteq \mathbb{R}^{\subconstituent{}}$, and that $\overline{W}_{\! k,k}$ is (essentially) the adjacency matrix of the \emph{$q$-Kneser graph} $qK_{n:k}$.
For $k,\ell = 0,\dots,n$, we define $J_{k,\ell} = I_k J I_{\ell}$, where $J \in \mathbb{R}^{\subconstituent{} \times \subconstituent{}}$ denotes the all ones matrix.

Let $k,\ell=1,\dots,\lfloor\frac{n}{2}\rfloor$ and assume for the moment that $k>\ell$.\footnote{\label{k=l} It should be remarked that the assumption $k > \ell$ is just for notational convenience. When $k=\ell$, we understand that $\subconstituent{\ell}$ is a distinct copy of $\subconstituent{k}$, and work instead with matrices with rows and columns indexed by $\subconstituent{} \cup \subconstituent{\ell}$.}
Let $\mathscr{F}\subseteq \subconstituent{k}$ and $\mathscr{G} \subseteq \subconstituent{\ell}$ be a pair of (nonempty) cross-intersecting families.
Let $\varphi, \psi \in \mathbb{R}^{\subconstituent{}}$ be the characteristic vectors of $\mathscr{F}$ and $\mathscr{G}$, respectively.
Then it follows that the matrix
\begin{equation*}
	X_{\mathscr{F}\! , \mathscr{G}} = \left(\frac{\varphi}{||\varphi||} + \frac{\psi}{||\psi||} \right) \!\! \left(\frac{\varphi}{||\varphi||} + \frac{\psi}{||\psi||} \right)^{\!\! \mathsf{T}} \in S \mathbb{R}^{\subconstituent{}\times \subconstituent{}}
\end{equation*}
is a feasible solution to the following SDP problem with objective value $|\mathscr{F}|^{\frac{1}{2}} |\mathscr{G}|^{\frac{1}{2}}$:
\begin{equation*}
\begin{array}{lll}
	\text{(P):} & \text{maximize} & \dfrac{1}{2} \left( J_{k,\ell} + J_{\ell,k} \right) \bullet X \\[0.1in]
		& \text{subject to} & I_k \bullet X = I_{\ell} \bullet X = 1, \\
		&& \left( \overline{W}_{\! k,\ell} + \overline{W}_{\! \ell,k} \right) \bullet X = 0, \\
		&& X \succcurlyeq 0, \ X \geqslant 0,
\end{array}
\end{equation*}
where $X \in S\mathbb{R}^{\subconstituent{} \times \subconstituent{}}$ is the variable, and $X \succcurlyeq 0$ (resp.~$X \geqslant 0$) means that $X$ is positive semidefinite (resp.~nonnegative).
The dual problem is then given by
\begin{equation*}
\begin{array}{lll}
	\text{(D):} & \text{minimize} & \alpha + \beta \\
		& \text{subject to} & S:= \alpha I_k + \beta I_{\ell} - \dfrac{1}{2} \left( J_{k,\ell} + J_{\ell,k} \right) \\[0.1in]
		&& \qquad \qquad \qquad - \gamma \left( \overline{W}_{\! k,\ell} + \overline{W}_{\! \ell,k} \right) - A \succcurlyeq 0, \\
		&& A \geqslant 0,
\end{array}
\end{equation*}
where $\alpha,\beta,\gamma\in\mathbb{R}$ and $A\in S\mathbb{R}^{\subconstituent{}\times\subconstituent{}}$ are the variables.\footnote{From the constraints it follows that $A_{x,y}=0$ unless $x,y\in\subconstituent{k}\cup\subconstituent{\ell}$.}
Indeed, for any feasible solutions to (P) and (D), we have
\begin{align*}
	\alpha+\beta - \frac{1}{2} \left( J_{k,\ell} + J_{\ell,k} \right) \bullet X &= \left(\alpha I_k + \beta I_{\ell} - \frac{1}{2} \left( J_{k,\ell} + J_{\ell,k} \right) \right) \bullet X \\
	&\geqslant \left( \gamma \left( \overline{W}_{\! k,\ell} + \overline{W}_{\! \ell,k} \right) + A \right) \bullet X \\
	&= A \bullet X \\
	&\geqslant 0.
\end{align*}
In particular, $(\alpha+\beta)^2$ gives an upper bound on $|\mathscr{F}| |\mathscr{G}|$.
Furthermore, notice that if $\alpha+\beta=\frac{1}{2} \left( J_{k,\ell} + J_{\ell,k} \right) \bullet X$ then $S\bullet X=A\bullet X=0$.
For the rest of this paper, we shall consider the following one-parameter family:
\begin{equation}\label{one-parameter family}
	\alpha=\beta=\frac{1}{2} \gauss{n-1}{k-1}^{\frac{1}{2}} \gauss{n-1}{\ell-1}^{\frac{1}{2}}, \quad \gamma=b(\lambda), \quad A=a (\lambda) \, \overline{W}_{\! k,k} + \lambda \, \overline{W}_{\! \ell,\ell},
\end{equation}
where $\lambda\in\mathbb{R}$, and
\begin{align*}
	q^{k^2} (q^k-1) \gauss{n-k}{k} a (\lambda) =& \, \frac{1}{2} q^{\ell} ( q^{k-\ell} - 1) \gauss{n-1}{k-1}^{\frac{1}{2}} \gauss{n-1}{\ell-1}^{\frac{1}{2}} \\
	& \qquad + q^{\ell^2} (q^{\ell}-1) \gauss{n-\ell}{\ell} \lambda, \\
	q^{k\ell} \gauss{n-k}{\ell} b(\lambda) =& - \frac{1}{2} q^{\ell} \gauss{n-1}{\ell} - q^{\ell^2} \gauss{n-\ell}{\ell} \gauss{n-1}{\ell-1}^{\frac{1}{2}} \gauss{n-1}{k-1}^{-\frac{1}{2}} \lambda.
\end{align*}
In Section \ref{sec: optimal feasible solutions}, we shall prove the following theorem:

\begin{thm}\label{feasible solutions}
For sufficiently small $\lambda >0$, \eqref{one-parameter family} gives a feasible solution to (D).
\end{thm}

\noindent
Theorem \ref{cross-intersection theorem} quickly follows from Theorem \ref{feasible solutions}, together with the above comments:

\begin{proof}[Proof of Theorem \ref{cross-intersection theorem}]
It follows that $|\mathscr{F}| |\mathscr{G}|\leqslant (\alpha+\beta)^2=\gauss{n-1}{k-1} \gauss{n-1}{\ell-1}$.
If equality holds, then $0=A\bullet X_{\mathscr{F}\! , \mathscr{G}} = a (\lambda) \, \overline{W}_{\! k,k} \bullet X_{\mathscr{F}\! , \mathscr{G}} + \lambda \, \overline{W}_{\! \ell,\ell} \bullet X_{\mathscr{F}\! , \mathscr{G}}$.
However, since $a(\lambda)>0$ whenever $\lambda>0$, it follows that $\overline{W}_{\! k,k} \bullet X_{\mathscr{F}\! , \mathscr{G}}=\overline{W}_{\! \ell,\ell} \bullet X_{\mathscr{F}\! , \mathscr{G}}=0$ in this case.
In other words, each of $\mathscr{F}$ and $\mathscr{G}$ is an intersecting family.
From the $q$-Erd\H{o}s--Ko--Rado theorem (cf.~\cite{Newman2004D,GN2006C,Tanaka2006JCTA}) it follows that $|\mathscr{F}|=\gauss{n-1}{k-1}$ and $|\mathscr{G}|=\gauss{n-1}{\ell-1}$, and the description of $\mathscr{F}$ and $\mathscr{G}$ also easily follows from that theorem.
\end{proof}

\begin{rem}
It is possible to consider SDP problems likewise for cross $t$-intersecting families (that is to say, we require that $\dim (x\cap y) \geqslant t$ for all $x\in \mathscr{F}$ and $y\in \mathscr{G}$).
However, constructing appropriate optimal feasible solutions for $t\geqslant 2$ appears to be a quite complicated problem, except when $k=\ell$, in which case it turns out that the problem reduces to LP.
See also \cite{Tokushige2012pre}.
We shall discuss cross $t$-intersecting families in future papers.
\end{rem}

\section{Block diagonalization}

Frankl and Wilson \cite{FW1986JCTA} obtained the bound in the $q$-Erd\H{o}s--Ko--Rado theorem.
This is an application of Delsarte's LP bound, and involves (among other results) calculations of the eigenvalues of the $q$-Kneser graphs.
In this section, we shall follow their method to block-diagonalize the matrices defining (P) and (D).

For $i,k,\ell=0,\dots,n$, it follows that (\cite[pp.~231--232]{FW1986JCTA})\footnote{For integers $a$ and $b$, we interpret $\gauss{a}{b}=0$ if $a<0$ or $b<0$.}
\begin{gather}
	W_{k,\ell} W_{\ell,i} = \gauss{k-i}{\ell-i} W_{k,i} \quad \text{if} \ i\leqslant \ell\leqslant k, \label{FW 4.2} \\
	W_{i,k} \overline{W}_{\! k,\ell} = q^{\ell(k-i)} \gauss{n-i-\ell}{k-i} \overline{W}_{\! i,\ell} \quad \text{if} \ i\leqslant k, \label{FW 4.1} \\
	\overline{W}_{\! k,\ell} = \sum_{h=0}^{ \min\{ k,\ell \} } (-1)^h q^{\binom{h}{2}} W_{k,h} W_{h,\ell}. \label{FW 4.6}
\end{gather}
We define subspaces $\harm{i}$ $(i=0,\dots,\lfloor \frac{n}{2}\rfloor)$ of $\mathbb{R}^{\subconstituent{}}$ as follows:
\begin{equation*}
	\harm{i} = \{ \bm{u}\in\mathbb{R}^{\subconstituent{i}} :  W_{i-1,i} \bm{u} =0 \},
\end{equation*}
where $W_{-1,0}:=0$.
Then

\begin{lem}\label{Wkiu}
For $i=0,\dots,\lfloor \frac{n}{2}\rfloor$, $k=0,\dots,n$ and $\bm{u}\in \harm{i}$, we have $W_{k,i} \bm{u} =0$ if $k<i$ or $k>n-i$.
Moreover,
\begin{equation*}
	\overline{W}_{\! k,i} \bm{u} = (-1)^i q^{\binom{i}{2}} W_{k,i} \bm{u}.
\end{equation*}
\end{lem}

\begin{proof}
From \eqref{FW 4.2} it follows that $W_{k,i} \bm{u} =0$ if $k<i$.
The other assertions follow from this and \eqref{FW 4.6}.
\end{proof}

The next two lemmas are easy consequences of Lemma \ref{Wkiu} and \eqref{FW 4.1}:

\begin{lem}\label{inner product}
For $i,j=0,\dots,\lfloor \frac{n}{2}\rfloor$, $k=0,\dots,n$, $\bm{u}\in \harm{i}$ and $\bm{v}\in \harm{j}$, we have
\begin{equation*}
	\bm{u}^{\mathsf{T}} W_{i,k} W_{k,j} \bm{v} = \delta_{i,j} q^{i(k-i)} \gauss{n-2i}{k-i} \bm{u}^{\mathsf{T}} \bm{v}.
\end{equation*}
\end{lem}

\begin{lem}\label{eigenvalue}
For $i=0,\dots,\lfloor \frac{n}{2}\rfloor$, $k,\ell=0,\dots,n$ and $\bm{u}\in \harm{i}$, we have
\begin{equation*}
	\overline{W}_{\! k,\ell} W_{\ell,i} \bm{u} = (-1)^i q^{\binom{i}{2}+k(\ell-i)} \gauss{n-k-i}{\ell-i} W_{k,i} \bm{u}.
\end{equation*}
\end{lem}

For $i=0,\dots,\lfloor \frac{n}{2} \rfloor$, we fix an orthonormal basis $\bm{u}_{i,1}, \bm{u}_{i,2}, \dots, \bm{u}_{i,d_i}$ of $U_i$, where $d_i:=\dim \harm{i} \ \big(\!\! =\gauss{n}{i} -\gauss{n}{i-1} \big)$.
We moreover define
\begin{equation*}
	\bm{u}_{i,r}^k = q^{-\frac{i(k-i)}{2}} \gauss{n-2i}{k-i}^{-\frac{1}{2}} W_{k,i} \bm{u}_{i,r} \quad (r=1,\dots,d_i, \, k=i,\dots,n-i).
\end{equation*}
From Lemmas \ref{inner product} and \ref{eigenvalue} it follows that

\begin{prop}\label{block-diagonalization}
The $\bm{u}_{i,r}^k$ form an orthonormal basis of $\mathbb{R}^{\subconstituent{}}$.
Furthermore, for $i=0,\dots,\lfloor \frac{n}{2} \rfloor$, $\ell=i,\dots,n-i$, $k=0,\dots,n$ and $r=1,\dots,d_i$, we have
\begin{equation*}
	\overline{W}_{\! k,\ell} \bm{u}_{i,r}^{\ell} = \theta_i^{k,\ell} \bm{u}_{i,r}^k,
\end{equation*}
where $\bm{u}_{i,r}^k:=0$ if $k<i$ or $k>n-i$, and
\begin{equation*}
	\theta_i^{k,\ell} = (-1)^i q^{\binom{i}{2} +k\ell -\frac{i(k+\ell)}{2} } \gauss{n-k-i}{\ell-i} \gauss{n-2i}{k-i}^{\frac{1}{2}} \gauss{n-2i}{\ell-i}^{-\frac{1}{2}}.
\end{equation*}
\end{prop}

\noindent
Notice that $\theta_i^{k,\ell}=0$ if $k<i$ or $k>n-i$, and that $\theta_i^{k,\ell}=\theta_i^{\ell,k}$ for $k,\ell=i,\dots,n-i$.
We also need the following observation:

\begin{lem}\label{J}
For $i=0,\dots,\lfloor \frac{n}{2} \rfloor$, $\ell=i,\dots,n-i$, $k=0,\dots,n$ and $r=1,\dots,d_i$, we have
\begin{equation*}
	J_{k,\ell} \bm{u}_{i,r}^{\ell} = \delta_{i,0} \gauss{n}{k}^{\frac{1}{2}} \gauss{n}{\ell}^{\frac{1}{2}} \bm{u}_{0,1}^k.
\end{equation*}
\end{lem}

\section{Proof of Theorem \ref{feasible solutions}}\label{sec: optimal feasible solutions}

In this section, we shall prove Theorem \ref{feasible solutions}.
Let $k,\ell=1,\dots,\lfloor\frac{n}{2}\rfloor$, and recall that we are assuming that $k>\ell$.\footnote{See footnote \ref{k=l}. The discussions in this section work for the case $k=\ell$ as well, and in fact become much simpler.}
Let $\lambda>0$.
Then $a(\lambda)>0$, so that \eqref{one-parameter family} gives a feasible solution to (D) if and only if the matrix
\begin{align*}
	S (\lambda) =& \, \frac{1}{2} \gauss{n-1}{k-1}^{\frac{1}{2}} \gauss{n-1}{\ell-1}^{\frac{1}{2}} \left( I_k + I_{\ell} \right) -\frac{1}{2} \left( J_{k,\ell} + J_{\ell,k} \right) \\
	& \qquad \qquad - a (\lambda) \, \overline{W}_{\! k,k} - \lambda \, \overline{W}_{\! \ell,\ell} - b (\lambda) \, \left( \overline{W}_{\! k,\ell} + \overline{W}_{\! \ell,k} \right)
\end{align*}
is positive semidefinite.
By virtue of Proposition \ref{block-diagonalization} and Lemma \ref{J}, $S(\lambda)\succcurlyeq 0$ if and only if
\begin{gather*}
	S_0 (\lambda) = \begin{pmatrix} \dfrac{1}{2} \dgauss{n-1}{k-1}^{\frac{1}{2}} \dgauss{n-1}{\ell-1}^{\frac{1}{2}} - \theta_0^{k,k} a(\lambda) & -\dfrac{1}{2} \dgauss{n}{k}^{\frac{1}{2}} \dgauss{n}{\ell}^{\frac{1}{2}} - \theta_0^{k,\ell} b(\lambda) \\ -\dfrac{1}{2} \dgauss{n}{k}^{\frac{1}{2}} \dgauss{n}{\ell}^{\frac{1}{2}} - \theta_0^{k,\ell} b(\lambda) & \dfrac{1}{2} \dgauss{n-1}{k-1}^{\frac{1}{2}} \dgauss{n-1}{\ell-1}^{\frac{1}{2}} - \theta_0^{\ell,\ell} \lambda \end{pmatrix} \succcurlyeq 0, \\
	S_i (\lambda) = \begin{pmatrix} \dfrac{1}{2} \dgauss{n-1}{k-1}^{\frac{1}{2}} \dgauss{n-1}{\ell-1}^{\frac{1}{2}} - \theta_i^{k,k} a(\lambda) & - \theta_i^{k,\ell} b(\lambda) \\ - \theta_i^{k,\ell} b(\lambda) & \dfrac{1}{2} \dgauss{n-1}{k-1}^{\frac{1}{2}} \dgauss{n-1}{\ell-1}^{\frac{1}{2}} - \theta_i^{\ell,\ell} \lambda \end{pmatrix} \succcurlyeq 0,
\end{gather*}
for $i=1,\dots,\ell$, and
\begin{equation*}
	s_i (\lambda) = \frac{1}{2} \gauss{n-1}{k-1}^{\frac{1}{2}} \gauss{n-1}{\ell-1}^{\frac{1}{2}} - \theta_i^{k,k} a(\lambda) \geqslant 0,
\end{equation*}
for $i=\ell+1,\dots,k$.

For $i=0,\dots,k$, we now calculate:
\begin{align}
	\left| \theta_i^{k,k} a(0) \right| =& \, \frac{1}{2} q^{\binom{i}{2} +\ell -ik} \frac{ q^{k-\ell} - 1 }{ q^k-1 } \gauss{n-k-i}{k-i} \gauss{n-k}{k}^{-1} \gauss{n-1}{k-1}^{\frac{1}{2}} \gauss{n-1}{\ell-1}^{\frac{1}{2}} \label{evaluation - diagonal} \\
	\leqslant & \, \frac{1}{2} q^{\ell - \frac{ik}{2} } \frac{ q^{k-\ell} - 1 }{ q^k-1 } \gauss{n-1}{k-1}^{\frac{1}{2}} \gauss{n-1}{\ell-1}^{\frac{1}{2}} \notag \\
	< & \begin{cases} \dfrac{1}{2} \dgauss{n-1}{k-1}^{\frac{1}{2}} \dgauss{n-1}{\ell-1}^{\frac{1}{2}} & \text{for} \ i=0,1, \\[.1in] \dfrac{1}{2q} \dgauss{n-1}{k-1}^{\frac{1}{2}} \dgauss{n-1}{\ell-1}^{\frac{1}{2}} & \text{for} \ i=2,\dots,k. \end{cases} \notag
\end{align}
Therefore the diagonal entries of the $S_i (\lambda)$ ($i=0,\dots,\ell$) as well as the $s_i (\lambda)$ ($i=\ell+1,\dots,k$) are positive, provided that $\lambda$ is sufficiently small.
Then it is readily verified that $S_0 (\lambda)$ and $S_1 (\lambda)$ are scalar multiples of the rank $1$ matrices\footnote{In fact, we arrived at this condition conversely with the help of the duality, i.e., that $S \bullet X_{\mathscr{F}\! ,\mathscr{G}}=0$ for `extremal' cross-intersecting families $\mathscr{F}$ and $\mathscr{G}$ given in Theorem \ref{cross-intersection theorem}. The constraints of (D) and the condition together led us to the one-parameter family \eqref{one-parameter family}.}
\begin{equation*}
	\begin{pmatrix} q^{\ell}-1 & -(q^k-1)^{\frac{1}{2}} (q^{\ell}-1)^{\frac{1}{2}} \\ -(q^k-1)^{\frac{1}{2}} (q^{\ell}-1)^{\frac{1}{2}} & q^k-1 \end{pmatrix}
\end{equation*}
and
\begin{equation*}
	\begin{pmatrix} q^{\ell} ( q^{n-\ell}-1 ) & -q^{\frac{k+\ell}{2}} (q^{n-k}-1)^{\frac{1}{2}} (q^{n-\ell}-1)^{\frac{1}{2}} \\ -q^{\frac{k+\ell}{2}} (q^{n-k}-1)^{\frac{1}{2}} (q^{n-\ell}-1)^{\frac{1}{2}} & q^k ( q^{n-k}-1 ) \end{pmatrix},
\end{equation*}
respectively, from which it follows that $S_0 (\lambda) \succcurlyeq 0$ and $S_1 (\lambda) \succcurlyeq 0$.
Hence it is enough to show that $\det ( S_i (0) )>0$ for $i=2,\dots,\ell$, i.e.,
\begin{equation*}
	\frac{1}{2} \gauss{n-1}{k-1}^{\frac{1}{2}} \gauss{n-1}{\ell-1}^{\frac{1}{2}} \left( \frac{1}{2} \gauss{n-1}{k-1}^{\frac{1}{2}} \gauss{n-1}{\ell-1}^{\frac{1}{2}} - \theta_i^{k,k} a(0) \right) > \left( \theta_i^{k,\ell} b(0) \right)^2
\end{equation*}
for $i=2,\dots,\ell$.
On one hand, it follows from \eqref{evaluation - diagonal} that the left-hand side is strictly larger than $\frac{1}{4} ( 1-\frac{1}{q} )\gauss{n-1}{k-1} \gauss{n-1}{\ell-1}$.
On the other hand, since
\begin{equation*}
	\left( \frac{ \theta_{i+1}^{k,\ell} }{ \theta_i^{k,\ell} } \right)^2 = \frac{ q^{ 2i-k-\ell } ( q^{k-i} -1 ) ( q^{\ell-i} -1 ) }{ ( q^{n-k-i} -1 ) ( q^{n-\ell-i} -1 ) } < 1 \quad (i=2,\dots,\ell-1),
\end{equation*}
it follows that
\begin{align*}
	\left( \theta_i^{k,\ell} b(0) \right)^2 \leqslant & \left( \theta_2^{k,\ell} b(0) \right)^2 \\
	=& \, \frac{1}{4} q^{2-2k} \frac{ q^{k-1}-1 }{ q^{n-k-1}-1 }  \frac{ q^{\ell-1}-1 }{ q^{n-\ell-1} -1 } \frac{ q^{n-\ell}-1 }{ q^{n-k}-1 } \gauss{n-1}{k-1} \gauss{n-1}{\ell-1} \\
	\leqslant & \, \frac{1}{4} q^{2-2k} \frac{ q^{n-\ell}-1 }{ q^{n-k}-1 } \gauss{n-1}{k-1} \gauss{n-1}{\ell-1} \\
	< & \, \frac{1}{4} q^{ (2-2k) + (k-\ell+1) } \gauss{n-1}{k-1} \gauss{n-1}{\ell-1} \\
	\leqslant & \, \frac{1}{4q} \gauss{n-1}{k-1} \gauss{n-1}{\ell-1}
\end{align*}
for $i=2,\dots,\ell$.
Hence $\det ( S_i(0) )>0$ for $i=2,\dots,\ell$, and the proof is complete.

\section*{Acknowledgements}
The authors thank Norihide Tokushige for valuable discussions and comments.
SS was supported in part by JSPS Research Fellowships for Young Scientists.
HT was supported in part by JSPS Grant-in-Aid for Scientific Research No.~23740002.

\end{document}